\documentclass[11pt]{aptpub}

\usepackage{amsfonts}
 \usepackage{mathrsfs}
 \usepackage{mathrsfs,amsfonts}
 \usepackage{color}
 \usepackage{amsfonts,amssymb,amsmath,mathrsfs}
 \def\prf{\noindent{\bf Proof.~~}}
 \def\deprf{\hfill$\Box$\medskip}

 \def\<{\langle}\def\>{\rangle}

 \def\mrm{\mathrm}

 \def\d{\mrm{d}}

\authornames{Y.H. Mao and T. Wang} 
\shorttitle{Non-strong Ergodicity of Markov Process} 

\begin{document}

\title{Lyapunov-type Conditions for Non-strong \\ Ergodicity of Markov Processes} 

\bigskip

\authorone[Beijing Normal University]{Yong-Hua Mao and Tao Wang} 

\addressone{School of Mathematical Sciences, Laboratory of Mathematics and Complex Systems, Ministry of Education, Beijing 100875, People’s Republic of China. \par Corresponding author: Tao Wang; E-mail address: wang\_tao@mail.bnu.edu.cn}

\begin{abstract}
We present Lyapunov-type conditions for non-strong ergodicity of Markov processes. Some concrete models are discussed including diffusion processes on Riemannian manifolds and Ornstein-Uhlenbeck processes driven by symmetric $\alpha$-stable processes. For SDE driven by $\alpha$-stable process ($\alpha\in (0,2]$) with polynomial drift, the strong ergodicity or not is independent on $\alpha$.
\end{abstract}


\section{Introduction} 
Strong ergodicity of Markov process is an important topic in ergodic theory for stochastic processes. Lyapunov criteria (drift conditions) for strong ergodicity have been discussed to obtain the sufficient conditions for strong ergodicity of Markov processes, see \cite{cmf04}, \cite{czq2014}, \cite{Meyn-tweedie1995}, \cite{MT1993}, \cite{wj2013}.
However, to obtain the necessary condition, we have to prove that Lyapunov functions do not exist. This is usually impractical, so we hope to find a sufficient (Lyapunov) condition for non-strong ergodicty.

It is well known that for right continuous Markov processes, strong ergodicity means the uniform boundedness of first moment of hitting time.
Our technique is based on this criteria and the martingale formulation.

The main results are two-fold. First, the Lyapunov-type conditions for null-recurrence of Markov processes are obtained by using two Lyapunov functions, see \cite{FMM 1995} for Markov chains and \cite{Stramer-Tweedie1994} for general Markov processes. We extend the method to non-strong ergodicity. Second,
motivated by Green function, we also obtain sufficient condition by using a Lyapunov function sequence.

Let $(X_t)_{t\geqslant 0}$ be a Markov process on a Polish space $(E,\mathscr{E})$ with transition function $P_t(x,\cdot)$. A $\sigma$-finite measure $\pi$ on $(E,\mathscr{E})$ is called an \textit{invariant measure} for $(X_t)_{t\geqslant 0}$, if for all $t>0$ and $A\in \mathscr{E}$, $\pi(A)=\int_{E}P_t(x,A)\pi(\d x)$. The process is called \textbf{ergodic}, if there exists a unique invariant probability measure $\pi,$ such that for all $x \in E$,
$\lim\limits_{t\rightarrow\infty}\|P_t(x,\cdot)-\pi(\cdot)\|_{\mathrm{Var}}=0,$ where $\|\cdot\|_{\mathrm{Var}} $ denotes total variation distance.

In this paper, we are interested in the strong ergodicity.
$(X_t)_{t\geqslant 0}$ is said to be \textbf{strongly ergodic (or uniformly ergodic)}, if there exist $\varepsilon>0$, a constant $C>0$, and a invariant probability measure $\pi,$ such that for all $t\geqslant 0$,
$$\sup\limits_{x\in E}\|P_t(x,\cdot)-\pi(\cdot)\|_{\mathrm{Var}}\leqslant C \mathrm{e}^{-\varepsilon t}.$$

From now on, we suppose that $(X_t)_{t\geqslant 0}$ is a time-homogeneous right continuous Markov process and evolves on a probability space $(\Omega, \mathscr{F}, (\mathscr{F}_t)_{t\geqslant0},\mathbb{P})$ with natural filtration $(\mathscr{F}_t)_{t\geqslant0}$. Assume that $(X_t)_{t\geqslant 0}$ is progressive measurable, then for any $A\in \mathscr{E}$, its hitting time $\tau_{A}:=\inf\{t\geqslant 0: X_t\in A\}$ is a stopping time with respect to $(\mathscr{F}_t)_{t\geqslant0}$.

We recall several notions we are going to use in our main results. Let $\{E_n\}_{n=1}^{\infty}\subset \mathscr{E}$ be a sequence of bounded open sets such that
\begin{equation}\label{En}
E_n\uparrow E, \ \bigcup_nE_n=E,
\end{equation}

Let $L$ be the \textit{infinitesimal generator}  of the process $(X_t)_{t\geqslant 0}$ with the domain $D(L)$ is given by
\begin{equation*}
\begin{split}
D(L):=\bigg\{&V : (E, \mathscr{E})\rightarrow (\mathbb{R},\mathscr{B} )\  \text{is measurable:}\ \lim_{h\rightarrow0}\frac{\mathbb{E}_x[V(X_h)]-V(x)}{h}\ \text{exists pointwise,} \\
&\text{and satisfies }\ \lim\limits_{h\rightarrow0}\mathbb{E}_x[LV(X_h)]=LV(x)\bigg\}.\\
\end{split}
\end{equation*}
Following \cite{Meyn-tweedie1995}, we use an enlarged domain of $L$ as follows:
\begin{equation*}
\begin{split}
D_w(L):=\bigg\{&f :  (E, \mathscr{E})\rightarrow (\mathbb{R},\mathscr{B} )\  \text{is measurable, } \ \text{such that}\ f(X_t)-f(X_0)-\int_{0}^{t}Lf(X_s)\d s\\
&  \text{is a local martingale}\bigg\}.\\
\end{split}
\end{equation*}

We say $(L,D_w(L))$ is the \textbf{extended generator} of $(X_t)_{t\geqslant 0}$. According to \cite[p.522]{Meyn-tweedie1993}, $D(L)\cap\{f: Lf \ \text{is locally bounded}\}\subset D_w(L)$.

A measurable function $f: E\rightarrow\mathbb{R}_+$ is called a \textbf{norm-like function (or compact function)}, if $f(x)\rightarrow\infty$ as $x\rightarrow\infty$; it means that the level sets $\{x: f(x)\leqslant r\}$ are precompact for each $r>0$.

Now we state the main results of this paper.

\begin{theorem}\label{inequality-type2}
Assume that $(X_t)_{t\geqslant 0}$ is non-explosive and ergodic. Let $\{E_n\}_{n=1}^{\infty}\subset \mathscr{E}$ be defined in (\ref{En}), and $H\subset E_1$ be a closed set with $\pi(H)>0$. If there exist non-negative norm-like functions $u(x),v(x)\in D_w(L)$, such that

\rm{(a)} $\forall x\notin H, Lu(x)\geqslant -1,$ and $u$ is locally bounded;

\rm{(b)} there exists a constant $d > 0$ such that $Lv(x)\leqslant d \mathbf{1}_H(x) $;

\rm{(c)} $\lim\limits_{n\rightarrow\infty}\sup\limits_{x\notin E_n}u(x)/v(x)=0,$

\noindent then  $(X_t)_{t\geqslant 0}$ is non-strongly ergodic.
\end{theorem}

\begin{remark}
(1) Specially, let $E'$ be a countable state space, $H$ be a non-empty finite subset of $E'$, and $(X_t)_{t\geqslant 0}$ be a $Q$-process with an irreducible regular $Q$-matrix $Q=(q_{ij})_{i,j\in E'}$, its generator $L$ is defined as
$$(Lu)_i=\sum_{j\in E'}q_{ij}u_j, \ \text{for}\ u=(u_i)_{i\in E'}.$$
According to \cite{Meyn-tweedie1993},
$\mathcal{V}:=\{u=(u_i)_{i\in E'}: \ \text{for any }\ i\in E', \ u_i \ \text{is finite}\}\subset D_w(L).$
Applying Theorem \ref{inequality-type2} to $Q$-process, we have that if there exist  $u,\ v\in \mathcal{V}$, such that

\rm{(a')} $\forall i\notin H, \sum_{j\in\mathbb{Z}_+}q_{ij}u_j\geqslant -1,$ and $\sum_{j\in\mathbb{Z}_+}q_{ij}v_j\leqslant 0 $;

\rm{(b')} $\varlimsup\limits_{i\rightarrow\infty}u_i=\varlimsup\limits_{i\rightarrow\infty}v_i=\infty,$ and $\varlimsup\limits_{i\rightarrow\infty}u_i/v_i=0,$

\noindent then the $Q$-process is non-strongly ergodic.

(2) A different criteria for non-strong ergodicity of Markov chain is obtained in \cite[Theorem 1.15(2)]{MP2004} (they call non-strong ergodicity as ``implosion does not occur"). The criteria is proved by a semimartingale approach for Markov chain (see \cite[p.2396]{MP2004}) and the idea is different from our method. For Markov chains, these two methods are both applicable.
\end{remark}

Next, we investigate the first moment of hitting time by the Green function.
Let $D$ be a domain, 
and define $P_t^D(x, A):=\mathbb{P}_x[X_t\in A,\tau_{D^c}>t].$ Assume that $P_t^D(x, \cdot)$ exists density $p_t^D(x, y)$ with respect to invariant measure $\pi(\d x)$. Then we define the Green function on $D$ as
$$G_D(x,y)=\int_{0}^{\infty}p_t^D(x, y)\d t.$$
If $\mathbb{E}_x \tau_{D^c}<\infty$ (the condition is ensured by the existence of stationary distribution), then
\begin{equation}\label{green on D}
\begin{split}
u_D(x)&:=\int_DG_D(x,y)\pi(dy)=\int_D \int_{0}^{\infty}p_t^D(x, y)\d t\pi(dy)\\
&=\int_{0}^{\infty}\mathbb{P}_x(\tau_{D^c}>t)\d t=\mathbb{E}_x\tau_{D^c}.
\end{split}
\end{equation}
Let $H$ and $\{E_n\}_{n=1}^{\infty}\subset \mathscr{E}$ be as in Theorem \ref{inequality-type2}. Assume that $(X_t)_{t\geqslant 0}$ is ergodic,  then the Poisson equation
\begin{equation}\label{Poi}
\left\{
\begin{split}
&Lu_n(x)= -1, & \text{in}\ E_n\setminus H;\\
&u_n(x)=0,& \text{in}  \ E_n^c\cup  H\\
\end{split}
\right.
\end{equation}
has finite solution $u_n(x)=\mathbb{E}_x[\tau_H\wedge \tau_{E_n^c}]=\int_{E_n\setminus H}G_{E_n\setminus H}(x,y)\pi(dy)$.
Hence the solution of Poisson equation (\ref{Poi}) could be represented by Green function.
If $\sup\limits_{x\notin H}\varlimsup\limits_{n\rightarrow\infty}u_n(x)=\infty$, then
$$\sup\limits_{x\notin H}\mathbb{E}_x\tau_H=\sup\limits_{x\notin H}\varlimsup\limits_{n\rightarrow\infty}\mathbb{E}_x[\tau_H\wedge \tau_{E_n^c}]=\sup\limits_{x\notin H}\varlimsup\limits_{n\rightarrow\infty}u_n(x)=\infty,$$
therefore by \cite[Lemma 2.1]{mao2002},
the process is non-strongly ergodic.

Motivated by this fact, we have the following result:
\begin{theorem}\label{Green criteria}
Assume that $(X_t)_{t\geqslant 0}$ is non-explosive and ergodic. Let $H$ and $\{E_n\}_{n=1}^{\infty}\subset \mathscr{E}$ be as Theorem \ref{inequality-type2}. If for each $n\geqslant1$, there exists a non-negative function $u_n(x)\in D_w(L)$, such that

\rm{(a)} $\forall x\in E_n\setminus H,$ $Lu_n(x)\geqslant -1;$

\rm{(b)} $u_n(x)=0$, \text{in}  $ E_n^c\cup  H$, and $u_n$ is bounded in  $E_n\setminus H;$

\rm{(c)} $\sup\limits_{x\notin H}\varlimsup\limits_{n\rightarrow\infty}u_n(x)=\infty,$

\noindent then $(X_t)_{t\geqslant 0}$ is non-strongly ergodic.
\end{theorem}

As a first step of our applications, let us check three simple examples. The following Corollaries 1, 2 and 3  are known and our new methodology is applied to reproduce these results as motivation.
\begin{corollary}[Diffusion process on half line]\label{diffusion process on half line}
Let $L=a(x)\frac{\mathrm{d}^2}{\mathrm{d}x^2}+b(x)\frac{\mathrm{d}}{\mathrm{d}x}$, $a(x)>0$ and $a,b$ be continuous on $(0,\infty)$.
Define $C(x)=\int_{1}^{x}(b(y)/a(y))\mathrm{d}y$ and $m(\d x)={a(x)}^{-1}\mathrm{e}^{C(x)}\d x$. Suppose the $L$-diffusion process $(X_t)_{t\geqslant 0}$ on $[0,\infty)$ with reflecting boundary at 0 is non-explosive and ergodic, i.e.
$$\int_{0}^{\infty}\mathrm{e}^{-C(y)}\left(\int_{0}^{y}\frac{\mathrm{e}^{C(z)}}{a(z)}\mathrm{d}z\right)\mathrm{d}y=\infty,\ \text{and} \ m([0,\infty))<\infty.$$

According to \cite[Theorem 2.1]{mao2002}, the process is strongly ergodic if and only if $$\delta:=\int_{0}^{\infty}\mathrm{e}^{-C(y)}\left(\int_{y}^{\infty}\frac{\mathrm{e}^{C(z)}}{a(z)}\mathrm{d}z\right)\mathrm{d}y<\infty.$$
\end{corollary}
\prf
Here we only consider the necessity. In order to apply Theorem \ref{inequality-type2}, let $u(x)=\int_{0}^{x}\mathrm{e}^{-C(y)}\left(\int_{y}^{\infty}{a(z)}^{-1}\mathrm{e}^{C(z)}\mathrm{d}z\right)\mathrm{d}y$ and $v(x)=\int_{0}^{x}\mathrm{e}^{-C(y)}\d y$. Then $Lu=-1$, $Lv=0$, and
$$\lim\limits_{x\rightarrow\infty}\frac{u(x)}{v(x)}=\lim\limits_{x\rightarrow\infty}\int_{x}^{\infty}\frac{\mathrm{e}^{C(y)}}{a(y)}\mathrm{d}y=0.$$
By Theorem \ref{inequality-type2}, we see that the process is non-strongly ergodic.

Now we apply Theorem \ref{Green criteria}. Let $E_n=(0,n), \ H=(0,1)$. The Green function on $E_n\setminus H$
$$G_{1,n}(x,y)=\frac{(s(n)-s(x\vee y))(s(x\wedge y)-s(1))}{s(n)-s(1)},$$
where $s(x)=\int_{0}^{x}\mathrm{e}^{-C(y)} \d y$. Let $u_n(x)=\int_{1}^{n}G_{1,n}(x,y)m(\d y)$.
For $x>1,$
\[
\begin{split}
\lim\limits_{n\rightarrow\infty}u_n(x)&=\int_{1}^{x}\mathrm{e}^{-C(t)}\left[\int_{t}^{x}\frac{1}{a(l)}\mathrm{e}^{C(l)}dl\right]dt+(s(x)-s(1))\int_{x}^{\infty}\frac{1}{a(l)}\mathrm{e}^{C(l)}dl\\
     &=\int_{1}^{x}\mathrm{e}^{-C(t)}\left[\int_{t}^{\infty}\frac{1}{a(l)}\mathrm{e}^{C(l)}dl\right]dt.
\end{split}
\]
Hence when $\delta=\infty,$ according to Theorem \ref{Green criteria}, the process is non-strongly ergodic. \deprf
\begin{corollary}[Single birth processes]
Let $Q=(q_{ij})_{i,j\in\mathbb{Z}_+}$ is a single-birth $Q$-matrix, i.e. $q_{i,i+1}>0, q_{i,i+j}=0$ for $i\in \mathbb{Z}_+, j\geqslant2.$ Assume that $Q$ is totally stable and conservative: $q_i:=-q_{ii}=\sum_{j\neq i}q_{ij}<\infty.$
Define
$$q_n^{(k)}=\sum_{j=0}^{k}q_{nj}, \ \  F_n^{(n)}=1, \ \ F_n^{(i)}=\frac{1}{q_{n,n+1}}\sum_{k=i}^{n-1}q_n^{(k)}F_k^{(i)}, \ \ 0\leqslant i<n,$$
$$d_0=1, \ \ d_n=\frac{1}{q_{n,n+1}}\left(1+\sum_{k=0}^{n-1}q_n^{(k)}d_k\right), \ \ n\geqslant 1, \ \ d:=\sup_{k\in\mathbb{Z}_+}\frac{\sum_{n=0}^{k}d_n}{\sum_{n=0}^{k}F_n^{(0)}}.$$

Assume that the $Q$-process is ergodic, i.e. $d<\infty.$ According to \cite[Theorem 1.1]{zhang2001},
the $Q$-process is strongly ergodic if and only if $$\sup_{k\geqslant 0}\sum_{n=0}^{k}(F_n^{(0)}d-d_n)<\infty.$$
\end{corollary}
\prf Here we only consider the necessity.
Let $H=\{0\}$ and
$$u_k=\sum_{n=0}^{k-1}(F_n^{(0)}d-d_n), \ \ \ v_k=\sum_{n=0}^{k-1}F_n^{(0)}.$$
It is well known that if $d<\infty$, then $u,v$ satisfy that $\sum_{j\geqslant 0}q_{ij}u_j=-1, \ \sum_{j\geqslant 0}q_{ij}v_j=0$, for $i\notin H$.
By \cite[Remarks 2.3(ii)]{zhang2001},  $\lim\limits_{k\rightarrow\infty}{d_k}/{F_k^{(0)}}=d,$ hence
$\lim\limits_{k\rightarrow\infty}u_k/v_k=0.$ Therefore, the non-strong ergodicity follows from
Remark 1(1).
\deprf

\begin{corollary}
Let $Q=(q_{ij})_{i,j\geqslant 0}$ is a birth-death $Q$-matrix with
\begin{equation}
q_{ij}=\left\{
\begin{split}
&p_i, & \ j=i+1;\\
&q_i,&   \ j=i-1;\\
&q_0,&   \ j=i=0;\\
&0,&   \ \text{otherwise}.\\
\end{split}
\right.
\end{equation}

(1) If $p_i=q_i=i^{\alpha}$ with $\alpha>1$, then the process is ergodic; furthermore, if $\alpha\leqslant2$, then the process is not strongly ergodic.

(2) If $p_i\equiv p>0,$ $q_i\equiv q>p,$ then the process is ergodic but not strongly
ergodic.

(3) If $p_i\equiv p,\ q_i=i^\alpha$ with $\alpha\in (0,1]$, then the process is non-strongly ergodic.
\end{corollary}
\prf
(1) If $p_i=q_i=i^{\alpha}$, then the invariant measure $\mu=\sum_{i=0}^{\infty}p_0i^{-\alpha},$ hence it is ergodic iff $\alpha>1.$ Furthermore, it is strongly ergodic iff $S:=\sum_{i=0}^{\infty}\sum_{j=i+1}^{\infty}j^{-\alpha}<\infty$, i.e. $\alpha>2$ (see Theorem 3.1 in \cite{mao2002}). Specially, we can use Theorem 1 to check the non-strong ergodicity. Let $u_i=\log(i+1)$ and $v_i=i.$ Then $\lim\limits_{k\rightarrow\infty}u_k/v_k=0,$ and
$$(Lv)_i=\sum_{j\geqslant 0}q_{ij}v_j=0,\  \ (Lu)_i=\sum_{j\geqslant 0}q_{ij}u_j=i^\alpha\log[1-1/(i+1)^2].$$
If $\alpha\leqslant 2,$ then we can choose $N$ large enough, such that for $i>N$, $(Lu)_i\geqslant-1.$ Thus the condition in Remark 1(1) is satisfied, the process is non-strongly ergodic.

(2) If $p_i\equiv p,\ q_i\equiv q>p$, then  the invariant measure $\mu=q^{-1}p_0\sum_{i=0}^{\infty}\left({p}/{q}\right)^{n-1}<\infty,$ hence it is ergodic.

Let $H=\{0\},$ $u_i=\log(i+1)$ and $v_i=i.$ Then $\lim\limits_{k\rightarrow\infty}u_k/v_k=0,$
$$(Lv)_i=\sum_{j\geqslant 0}q_{ij}v_j=p-q<0,$$
and
$$ (Lu)_i=\sum_{j\geqslant 0}q_{ij}u_j=p\log\left(\frac{i+2}{i+1}\right)-q\log\left(\frac{i+1}{i}\right)\geqslant-q\log 2, \ \text{for all} \ i\notin H.$$
Thus the condition in Remark 1(1) is satisfied, the process is non-strongly ergodic.

(3) Let $u_i=\log(i+1)$, $v_i=i,$ and $H=\{0,\cdots,  p^{1/\alpha}\vee 1\}.$ Then $\lim\limits_{k\rightarrow\infty}u_k/v_k=0,$ and for $i\notin H,$
$$(Lv)_i=\sum_{j\geqslant 0}q_{ij}v_j=p-i^\alpha<0,$$
and
$$ (Lu)_i=\sum_{j\geqslant 0}q_{ij}u_j=p\log\left(\frac{i+2}{i+1}\right)-i^{\alpha}\log\left(\frac{i+1}{i}\right)\geqslant-i^{\alpha-1}\geqslant-1.$$
So the condition in Remark 1(1) is satisfied, the process is non-strongly ergodic.
\deprf

The remainder of this paper is organized as follows. In section \ref{General criteria for non-strong ergodicity}, we give proofs of Theorem \ref{inequality-type2} and Theorem \ref{Green criteria}. In section \ref{Diffusion processes} we present a new sufficient condition for non-strong ergodicity of diffusion process on Riemannian manifold and give some examples. In section \ref{Ornstein-Uhlenbeck processes driven by stable noises}, we prove the non-strong ergodicity of Ornstein-Uhlenbeck processes driven by symmetric $\alpha$-stable noises.

\section{General criteria for non-strong ergodicity}\label{General criteria for non-strong ergodicity}
In this section we prove Theorem \ref{inequality-type2} and \ref{Green criteria}. For this, we need the following result.

\begin{lemma}\cite[Lemma 2.1]{mao2002}\label{necessary condition by hitting time}
Let $(X_t)_{t\geqslant 0}$ be a right continuous Markov process on $(E,\mathscr{E})$. If $(X_t)_{t\geqslant 0}$ is strongly ergodic, then we have $\sup\limits_{x\in E}\mathbb{E}_x\tau_A<\infty,$ for any closed set $A\subset E$ with $\pi(A)>0.$
\end{lemma}

\noindent\textbf{\textbf{Proof of Theorem \ref{inequality-type2}.}}\
First, according to the condition (a) and (b), we have that for $t\geqslant 0$ and $x\in E_n\setminus H$,
\begin{equation}\label{martingale 1}
\mathbb{E}_x[u(X_{ t\wedge\tau_{H}\wedge \tau_{E_n^c}})]-u(x)=\mathbb{E}_x\left[\int_{0}^{t\wedge\tau_{H}\wedge \tau_{E_n^c}}Lu(X_s)\d s\right]\geqslant -\mathbb{E}_x[ t\wedge\tau_{H}\wedge \tau_{E_n^c}],
\end{equation}
and
$$\mathbb{E}_x[v(X_{ t\wedge\tau_{H}\wedge \tau_{E_n^c}})]-v(x)=\mathbb{E}_x\left[\int_{0}^{t\wedge\tau_{H}\wedge \tau_{E_n^c}}Lv(X_s)\d s\right]\leqslant 0.$$
According to Fatou's lemma, when $t\rightarrow \infty$, we have for $x\in E_n\setminus H$,
$$\mathbb{E}_x[v(X_{\tau_{H}\wedge \tau_{E_n^c}})]\leqslant \varliminf_{t\rightarrow\infty}\mathbb{E}_x[v(X_{t\wedge \tau_{H}\wedge \tau_{E_n^c}})]\leqslant v(x).$$
Next, note that
on $\{t<\tau_{H}\wedge \tau_{E_n^c}\}$, $X_t\in E_n\setminus H,$ hence by the local boundedness of $u$,
$$u(X_t)\mathbf{1}_{\{t<\tau_{H}\wedge \tau_{E_n^c}\}}\leqslant \sup_{x\in \overline{E_n\setminus H}}u(x):=M<\infty.$$
Thus by ergodicity,
\begin{equation}
0\leqslant\lim_{t\rightarrow\infty}\mathbb{E}_x[u(X_t)\mathbf{1}_{\{t<\tau_{H}\wedge \tau_{E_n^c}\}}]\leqslant M \lim_{t\rightarrow\infty}\mathbb{P}_x[\tau_{H}\wedge \tau_{E_n^c}>t]=0.
\end{equation}
Therefore,
\begin{equation}\label{martingale 2}
\begin{split}
\lim_{t\rightarrow\infty}\mathbb{E}_x[u(X_{ t\wedge\tau_{H}\wedge \tau_{E_n^c}})]&=\lim_{t\rightarrow\infty}\mathbb{E}_x[u(X_{ \tau_{H}\wedge \tau_{E_n^c}})\mathbf{1}_{\{t\geqslant\tau_{H}\wedge \tau_{E_n^c}\}}]+\lim_{t\rightarrow\infty}\mathbb{E}_x[u(X_t)\mathbf{1}_{\{t<\tau_{H}\wedge \tau_{E_n^c}\}}]  \\
     & \leqslant\mathbb{E}_x[u(X_{ \tau_{H}\wedge \tau_{E_n^c}})].
\end{split}
\end{equation}
Now by combining (\ref{martingale 1}) and (\ref{martingale 2}), we have
\begin{equation}\label{martingale 3}
\mathbb{E}_x[\tau_{H}\wedge \tau_{E_n^c}]\geqslant u(x)-\mathbb{E}_x[u(X_{ \tau_{H}\wedge \tau_{E_n^c}})].
\end{equation}
Since
$$\mathbb{E}_x[u(X_{\tau_{H}\wedge \tau_{E_n^c}})]=\mathbb{E}_x[u(X_{\tau_{H}})\mathbf{1}_{\{\tau_{H}<\tau_{E_n^c}\}}]+\mathbb{E}_x[u(X_{\tau_{E_n^c}})\mathbf{1}_{\{\tau_{H}>\tau_{E_n^c}\}}]$$
and $\mathbb{E}_x[v(X_{\tau_{E_n^c}})\mathbf{1}_{\{\tau_{H}>\tau_{E_n^c}\}}]\leqslant v(x)$, it follows from (\ref{martingale 3}) that
\[
\begin{split}
 \mathbb{E}_x[\tau_{H}\wedge \tau_{E_n^c}]&\geqslant u(x)-\mathbb{E}_x[u(X_{\tau_{H}})]-\mathbb{E}_x\left[\frac{u(X_{\tau_{E_n^c}})}{v(X_{\tau_{E_n^c}})}v(X_{\tau_{E_n^c}})\mathbf{1}_{\{\tau_{H}>\tau_{E_n^c}\}}\right]\\
&\geqslant u(x)-\mathbb{E}_x[u(X_{\tau_{H}})]-\left(\sup_{x\notin E_n}\frac{u(x)}{v(x)}\right)v(x).\\
\end{split}
\]
Let $n\rightarrow\infty$ to derive $\sup\limits_{x\notin H}\mathbb{E}_x[\tau_{H}]\geqslant \sup\limits_{x\notin H}u(x)-\sup\limits_{x\in H}u(x)=\infty.$
Therefore, the process is non-strongly ergodic.  \deprf

\noindent\textbf{\textbf{Proof of Theorem \ref{Green criteria}.}}\
First, according to the condition (a) in Theorem \ref{Green criteria}, for $t\geqslant 0$ and $x\in E_n\setminus H$,
$$\mathbb{E}_x[u_n(X_{ t\wedge\tau_{H}\wedge \tau_{E_n^c}})]-u_n(x)=\mathbb{E}_x\left[\int_{0}^{t\wedge\tau_{H}\wedge \tau_{E_n^c}}Lu_n(X_s)\d s\right]\geqslant -\mathbb{E}_x[ t\wedge\tau_{H}\wedge \tau_{E_n^c}].$$
By similar argument as in the proof of Theorem \ref{inequality-type2}, we get
\begin{equation}
\begin{split}
\lim_{t\rightarrow\infty}\mathbb{E}_x[u_n(X_{ t\wedge\tau_{H}\wedge \tau_{E_n^c}})]&\leqslant\mathbb{E}_x[u_n(X_{ \tau_{H}\wedge \tau_{E_n^c}})]+\lim_{t\rightarrow\infty}\mathbb{E}_x[u_n(X_t)\mathbf{1}_{\{t<\tau_{H}\wedge \tau_{E_n^c}\}}]  \\
     & \leqslant\mathbb{E}_x[u_n(X_{ \tau_{H}\wedge \tau_{E_n^c}})]=0,
\end{split}
\end{equation}
where we use the condition (b) in the last term.
So for $x\in E_n\setminus H$,
$$\mathbb{E}_x[\tau_{H}\wedge \tau_{E_n^c}]\geqslant u_n(x).$$
Next, by letting $n\rightarrow\infty$, we have for all $x\notin H$, $\mathbb{E}_x[\tau_{H}]\geqslant \varlimsup\limits_{n\rightarrow\infty}u_n(x)$.
Therefore,
$$\sup_{x\notin H}\mathbb{E}_x[\tau_{H}]\geqslant \sup_{x\notin H}\varlimsup_{n\rightarrow\infty}u_n(x)=\infty.$$
This proves that $(X_t)_{t\geqslant 0}$ is non-strongly ergodic.
\deprf

\section{Diffusion processes}\label{Diffusion processes}
Let $M$ be a complete connected Riemannian manifold, $(X_t)_{t\geqslant 0}$ be a non-explosive and ergodic diffusion process on $M$ with generator $L=\Delta+Z$, where $Z$ is a $C^1$ vector field.
Assume that the generalized martingale problem for $L$ is well-posed, i.e. for $f\in C^2(M)$ and $Lf$ locally bounded,
$f(X_t)-f(X_0)-\int_{0}^{t}Lf(X_s)\d s $ is a local martingale with respect to $\mathbb{P}_x$, for any $ x\in M$. On a proper local chart of $M$, the genetor $L$ has the form
\begin{equation}\label{local chart}
L=\sum_{i,j=1}^{n}a_{ij}(x)\frac{\partial^2}{\partial x_i\partial x_j}+\sum_{i=1}^{n}b_i(x)\frac{\partial}{\partial x_i}.
\end{equation}
Specially, if $M=\mathbb{R}^n$, then the form (\ref{local chart}) is a global representation. If $a$ is positive definite, symmetric, and  $a, b$ are locally bounded, then  the martingale problem for $L$ is well-posed (see \cite[Theorem 1.13.1]{Pinsky1995}).

Let $\rho\in C^2(M\times M)$ be a distance (may not be Riemannian metric). Fix $o\in M$, let $\rho(x)=\rho(x,o)$, and
$D=\sup_x \rho(x)$ be the diameter, $B_d:=\{x\in M:\rho(x)\leqslant d\}$ be the geodesic ball.

For $\xi,\eta \in C^2(M),$ define $\Gamma(\xi,\eta)=\frac{1}{2}\left(L(\xi\eta)-\xi L\eta-\eta L\xi\right)$. If $\rho$ is the Riemannian distance on $M$, then $\Gamma(\rho,\rho)\equiv 1$. When $M=\mathbb{R}^n$, $\rho$ is the Euclidean distance and $L$ has the form (\ref{local chart}) satisfying that $a(x)$ is positive define, we have $\Gamma(\rho,\rho)=\frac{1}{|x|^2}\sum_{i,j=1}^{n}a_{ij}(x)x_ix_j>0$.

In this section, we always assume that the distance fucntion $\rho$ satisfies $\Gamma(\rho,\rho)>0.$

Define $\mathscr{F}=\{f\in C^2[0, D]: f\big|_{(0,D)}>0, f'\big|_{(0,D)}\geqslant 0\}$. For $f\in \mathscr{F}$,
\begin{equation}
Lf\circ\rho(x)=\Gamma(\rho,\rho)(x)f''[\rho(x)]+L\rho(x)f'[\rho(x)].
\end{equation}

Next, fix $0<p<D$, choose the functions as follows: for $r\geqslant p,$
\begin{equation}\label{alpha}
\overline{\alpha}(r)\geqslant\sup_{\rho(x)=r}\Gamma(\rho,\rho)(x), \ \  \underline{\alpha}(r)\leqslant\inf_{\rho(x)=r}\Gamma(\rho,\rho)(x);
\end{equation}
$$\overline{\beta}(r)\geqslant\sup_{\rho(x)=r}L\rho(x), \ \ \underline{\beta}(r)\leqslant \inf_{\rho(x)=r}L\rho(x);$$
\begin{equation}\label{I}
\overline{C}(r)=\int_{p}^{r}\frac{\overline{\beta}(s)}{\underline{\alpha}(s)}ds, \ \ \underline{C}(r)=\int_{p}^{r}\frac{\underline{\beta}(s)}{\overline{\alpha}(s)}ds.
\end{equation}
Then by comparing $(X_t)_{t\geqslant 0}$ with its radial process and applying Foster-Lyapunov criteria (see \cite[Theorem 5.2(c)]{Meyn-tweedie1993}) and Theorem \ref{Green criteria}, we obtain the explicit conditions for the strong ergodicity and the non-strong ergodicity of diffusion processes on manifolds.
\begin{theorem}\label{criteria}
(1) If
\begin{equation}\label{strong ergodicity condition}
\overline{\delta}_p(\rho):=\int_ {p}^{D}\mathrm{e}^{-\overline{C}(y)}\left(\int_ {y}^{D}\frac{\mathrm{e}^{\overline{C}(z)}}{\underline{\alpha}(z)}dz\right)dy<\infty,
\end{equation}
then the process $(X_t)_{t\geqslant 0}$ is strongly ergodic.

(2) If
\begin{equation}\label{non-strong ergodicity condition}
\underline{\delta}_p(\rho):=\int_ {p}^{D}\mathrm{e}^{-\underline{C}(y)}\left(\int_ {y}^{D}\frac{\mathrm{e}^{\underline{C}(z)}dz}{\overline{\alpha}(z)}\right)dy=\infty,
\end{equation}
then the process $(X_t)_{t\geqslant 0}$ is non-strongly ergodic.
\end{theorem}
To prove Theorem \ref{criteria}, we need the following lemma.
\begin{lemma}\label{lem}
Assume the diffusion $(X_t)_{t\geqslant 0}$ is non-explosive. If $s(D):=\int_{1}^{D}\mathrm{e}^{-\underline{C}(l)}dl<\infty,$
then $(X_t)_{t\geqslant 0}$ is transient.
\end{lemma}
\prf
For $r\geqslant 1,$ define $s(r):=\int_{1}^{r}\mathrm{e}^{-\underline{C}(l)}dl,$ then for $r\geqslant 1,$
$$\overline{\alpha}(r)s''(r)+\underline{\beta}(r)s'(r)=0, \ \ \text{i.e.} \ \ s''(r)+\frac{\underline{\beta}(r)}{\overline{\alpha}(r)}s'(r)=0.$$
Thus for $x\in E$ with $\rho(x)=r$,
\begin{equation}\label{Lu>0}
Ls\circ\rho(x)=A(x)\left(s''[\rho(x)]+\frac{B(x)}{A(x)}s'[\rho(x)]\right)\geqslant A(x)\left(s''(r)+\frac{\underline{\beta}(r)}{\overline{\alpha}(r)}s'(r)\right)=0.
\end{equation}
The martingale property implies that
$$\mathbb{E}_x[s\circ\rho(X_{t\wedge \tau_{B_{1}}\wedge \tau_{B_{R}^c}})]\geqslant s\circ\rho(x).$$
By letting $t\rightarrow\infty$, we have $s\circ\rho(x)\leqslant  (1-\mathbb{P}_x[\tau_{B_1}<\tau_{B_{R}^c}])s(R)$. So
$$\mathbb{P}_x[\tau_{B_1}<\tau_{B_{R}^c}]\leqslant \frac{s(R)-s(\rho(x))}{s(R)}, \ \text{for}\ 1\leqslant\rho(x)\leqslant R.$$
Because $(X_t)_{t\geqslant 0}$ is non-explosive, let $R\rightarrow D$ to get that for $x\in E$ with $\rho(x)\geqslant 1,$
$$\mathbb{P}_x[\tau_{B_1}<\infty]\leqslant \frac{s(D)-s(\rho(x))}{s(D)}<1.$$
Therefore $(X_t)_{t\geqslant 0}$ is transient.
\deprf

\noindent\textbf{\textbf{Proof of Theorem \ref{criteria}}}
\def\roman2{\uppercase\expandafter{\romannumeral2}}

(1) $\forall r\geqslant p$, define 
\begin{equation}\label{f1}
f_1(r)=\int_{p}^{r}\mathrm{e}^{-\overline{C}(y)}\left(\int_ {y}^{\infty}\frac{\mathrm{e}^{\overline{C}(z)}}{\underline{\alpha}(z)}dz\right)dy.
\end{equation}
Then $f_1$ satisfies that
$$\underline{\alpha}(r)f_1''(r)+\overline{\beta}(r)f_1'(r)=-1.$$
Hence for $x\in M$ with $\rho(x)=r$,
\begin{equation}\label{lyapunov1}
\begin{split}
Lf_1\circ\rho(x)&=\Gamma(\rho,\rho)(x)\left(f_1''[\rho(x)]+\frac{L\rho(x)}{\Gamma(\rho,\rho)(x)}f_1'[\rho(x)]\right)\\
&\leqslant \underline{\alpha}(r)f_1''(r)+\overline{\beta}(r)f_1'(r)=-1.
\end{split}
\end{equation}

If $\overline{\delta}_p(\rho)<\infty$, then by letting $u_1(x)=f_1\circ\rho(x)$, and $H=\overline{B_p}$, the process is strongly ergodic by \cite[Theorem 5.2(c)]{Meyn-tweedie1993}.

(2) Let $u_n(x)=\psi_n\circ\rho(x)$ be defined by
\begin{equation}\label{u_n}
\psi_n(r)=\int_ {p}^{n}G(r,l)\frac{1}{\overline{\alpha}(l)}\mathrm{e}^{\underline{C}(l)}dl, \ \ \  p\leqslant r\leqslant n,
\end{equation}
where
\begin{equation}
G(r,l):=\left\{
\begin{split}
&\frac{[s(r)-s(p)][s(n)-s(l)]}{s(n)-s(p)}& r<l;\\
&\frac{[s(l)-s(p)][s(n)-s(r)]}{s(n)-s(p)}& r\geqslant l.\\
\end{split}
\right.
\end{equation}
For $ n\geqslant p+1,$ let $E_n=B_{n+p}$.
Obviously, $\forall x\in E_n\setminus H,\ Lu_n(x)\geqslant -1,$ and $u_n\big|_{E_n\cup \partial H}=0.$
Rewrite $u_n(r)$ as
\[
\begin{split}\label{Green}
\psi_n(r)=&\int_{p}^{r}\frac{[s(l)-s(p)][s(n)-s(r)]}{s(n)-s(p)}\frac{1}{\overline{\alpha}(l)}\mathrm{e}^{\underline{C}(l)}dl+\int_{r}^{n}\frac{[s(r)-s(p)][s(n)-s(l)]}{s(n)-s(p)}\frac{1}{\overline{\alpha}(l)}\mathrm{e}^{\underline{C}(l)}dl.\\
\end{split}
\]
If $\lim\limits_{R\rightarrow D}s(R)<\infty$, then by letting $n\rightarrow D$, we know that $(X_t)_{t\geqslant 0}$ is transient by Lemma \ref{lem}. So we assume that $\lim\limits_{n\rightarrow \infty}s(n)=\infty$. Therefore,
\[
\begin{split}
\lim_{n\rightarrow\infty}\psi_n(r)&=\int_{p}^{r}\left[\int_{p}^{l}\mathrm{e}^{-\underline{C}(t)}dt\right]\frac{1}{\overline{\alpha}(l)}\mathrm{e}^{\underline{C}(l)}dl+\int_{p}^{r}\mathrm{e}^{-\underline{C}(t)}dt\int_{r}^{\infty}\frac{1}{\overline{\alpha}(l)}\mathrm{e}^{\underline{C}(l)}dl\\
     &=\int_{p}^{r}\mathrm{e}^{-\underline{C}(t)}\left[\int_{t}^{r}\frac{1}{\overline{\alpha}(l)}\mathrm{e}^{\underline{C}(l)}dl\right]dt+\int_{p}^{r}\mathrm{e}^{-\underline{C}(t)}\left[\int_{r}^{\infty}\frac{1}{\overline{\alpha}(l)}\mathrm{e}^{\underline{C}(l)}dl\right]dt\\
&=\int_{p}^{r}\mathrm{e}^{-\underline{C}(t)}\left[\int_{t}^{\infty}\frac{1}{\overline{\alpha}(l)}\mathrm{e}^{\underline{C}(l)}dl\right]dt.
\end{split}
\]
Then (\ref{non-strong ergodicity condition}) yields that $\sup\limits_{x\notin H}\lim\limits_{n\rightarrow \infty}u_n(x)=\sup\limits_{r\geqslant p}\lim\limits_{n\rightarrow\infty}\psi_n(r)=\infty.$ It follows from Theorem \ref{Green criteria} that $(X_t)_{t\geqslant 0}$ is non-strongly ergodic.
\deprf

We use Theorem \ref{criteria} to check some examples such as radial process.
\begin{corollary}\cite[Example 3.6]{mao2006}\label{nd simple example}
Let $(X_t)_{t\geqslant 0}$ be a $n$-dimensional diffusion process with generator $L=\Delta+\nabla V\cdot\nabla$. Here $V(x)=-|x|^c$. Then $(X_t)_{t\geqslant 0}$ is strongly ergodic if and only if $c>2.$ Specially, the classical O.U. process ($c=2$) is non-strongly ergodic.
\end{corollary}
\prf
Let $\rho(x)=|x|$. By Theorem \ref{criteria}, we know $(X_t)_{t\geqslant 0}$ is strongly ergodic if and only if
$$ \int_ {1}^{\infty}y^{1-n}\mathrm{e}^{y^c}\left(\int_ {y}^{\infty}z^{n-1}\mathrm{e}^{-z^c}dz\right)dy<\infty.$$
By using integration by parts, we obtain that there exist $C_1,C_2>0$ such that
\begin{equation}\label{integral by part}
C_1 y^{n-c}\mathrm{e}^{-y^c}\leqslant\int_ {y}^{\infty}z^{n-1}\mathrm{e}^{-z^c}dz\leqslant C_2 y^{n-c}\mathrm{e}^{-y^c}.
\end{equation}
Hence $(X_t)_{t\geqslant 0}$ is strongly ergodic if and only if
$$\infty>\int_ {1}^{\infty}y^{1-n}\mathrm{e}^{y^c}y^{n-c}\mathrm{e}^{-y^c}dy=\int_ {1}^{\infty}y^{1-c}dy, $$
which is equivalent to $ c>2.$

On the other hand, we can also use Lyapunov function for strong ergodicity and Theorem \ref{inequality-type2}.

When $c\leqslant0$, $(X_t)_{t\geqslant 0}$ is not ergodic. For $0<c\leqslant 2$, we choose $u(x)=\log (|x|+1)$, $v(x)=|x|$.  Then
 $$Lu\geqslant\frac{n-2}{(|x|+1)^2}-\frac{1}{|x|(|x|+1)^2}-c|x|^{c-2}, \  \text{and} \ Lv=\frac{n-1-c|x|^c}{|x|}.$$
Let $r=\left(n-1/c\right)^{1/c}$, $H=\overline{B_r}$ and $E_n=B_{r+n}$.
When $x\notin H, $
$$Lv(x)\leqslant 0, \ \text{and} \ Lu(x)\geqslant-\left(\frac{n-1}{c}\right)^{-\frac{3}{c}}-c\left(\frac{n-1}{c}\right)^{\frac{c-2}{c}}.$$
It is easy to check that $u,v$ satisfy the condition in Theorem \ref{inequality-type2}, so that $(X_t)_{t\geqslant 0}$ is non-strongly ergodic;

When $c>2$, let $w(x)=1-{\log(|x|+1)}^{-1}$ and it satisfies the Lyapunov condition for strong ergodicity (see \cite[Theorem 5.2(c)]{Meyn-tweedie1995}), therefore, $(X_t)_{t\geqslant 0}$ is strongly ergodic.
\deprf

Now we remark that Theorem \ref{criteria} is somewhat difficult to get the strong ergodicity for some non-radial processes.
To use Theorem \ref{criteria}, we need to choose a distance function $\rho$ to compare $(X_t)_{t\geqslant 0}$ with its racial process. However, for some processes, choosing the (smooth) distance function (such as Riemannian metric) is difficult to get the strong ergodicity.
However, Theorem \ref{inequality-type2} can still be valid.

\begin{corollary}\label{independent coefficient example}
Let $(X_t)_{t\geqslant 0}$ be a diffusion process on $\mathbb{R}^2$ with generator $L=\frac{\partial^2}{\partial x_1^2}+\frac{\partial^2}{\partial x_2^2}-x_1\frac{\partial}{\partial x_1}-x_2^2\frac{\partial}{\partial x_2}.$ Then $(X_t)_{t\geqslant 0}$  is non-strongly ergodic.
\end{corollary}
\prf
Let $\rho(x)=|x|$ be the Euclid metric. For $r>0$ large enough, we choose
$$\overline{\alpha}(r)=\underline{\alpha}(r)=1, \ \overline{\beta}(r)=\frac{1+r^3}{r}, \ \underline{\beta}(r)=\frac{1-r^3}{r}$$
and
$$\overline{C}(r)=\log r+\frac{r^3-1}{3}, \ \ \underline{C}(r)=\log r-\frac{r^3-1}{3}.$$
For $p>0$ large enough, on the one hand,
$$\overline{\delta}_p(\rho)=\int_ {p}^{\infty}y^{-1}\mathrm{e}^{-y^3/3}\left(\int_ {y}^{\infty}z\mathrm{e}^{z^3/3}dz\right)dy\geqslant \int_ {p}^{\infty}y^{-1}\left(\int_ {y}^{\infty}zdz\right)dy=\infty;$$
On the other hand, $\underline{\delta}_p(\rho)<\infty$ according to (\ref{integral by part}).
Thus Theorem \ref{criteria} is invalid to check the strong ergodicity.

Now we apply Theorem \ref{inequality-type2}. We choose the functions $u(x)=\log (x_1^2+1), $ $v(x)=x_1^2+1$, and $E_n=\{x: |x_1|< n, |x_2|< n\}$ for $n\geqslant 2$, and $H=\overline{E_1}$. For $ x\notin H,$
$$Lu=\frac{2(1-x_1^2)}{(1+x_1^2)^2}-\frac{2x_1^2}{1+x_1^2}\geqslant-\frac{5}{2}, \ Lv=2-2x_1^2\leqslant 0, \ \text{and}\ \lim\limits_{n\rightarrow\infty}\frac{u(x)}{v(x)}=0.$$
Thus the process is non-strongly ergodic.
\deprf

\section{Ornstein-Uhlenbeck processes driven by $\alpha$-stable noises}\label{Ornstein-Uhlenbeck processes driven by stable noises}
Let $(Z_t)_{t\geqslant 0}$ be a $d-$dimensional symmetric $\alpha-$stable process with generator $-(-\Delta)^{\alpha/2},$ which has the following expression:
$$-(-\Delta)^{\alpha/2}f(x):=\int_{\mathbb{R}^d\setminus\{0\}}\left(f(x+z)-f(x)-\nabla f(x)\cdot z\mathbf{1}_{\{|z|\leqslant 1\}}\right)\frac{C_{d,\alpha}}{|z|^{d+\alpha}} \d z.$$
Here $C_{d,\alpha}=\frac{\alpha2^{\alpha-1}\Gamma((d+\alpha)/2)}{\pi^{d/2}\Gamma(1-\alpha/2)}$ is the normalizing constant so that the Fourier transform 
of $-(-\Delta)^{\alpha/2}u$ is $-|\xi|^{\alpha}\hat{u}(\xi).$

Consider the following stochastic differential equation driven by $\alpha$-stable noise on $\mathbb{R}^d$:
$$\d X_t = AX_t \d t + \d Z_t, \  X_0 = x,$$
where  $A$ is a real $d \times d$ matrix. It is well known that the SDE has the unique strong solution $(X_t)_{t\geqslant 0}$ which is
(strong) Feller and Lebesgue irreducible, see, e.g. \cite{wj2013}. We call $(X_t)_{t\geqslant 0}$ $d$-dimensional Ornstein-Uhlenbeck process driven by symmetric $\alpha-$stable noise.
The generator $L$ is represented as $\text{for any} \ f\in D_w(L),$
$$Lf(x)=\int_{\mathbb{R}^d\setminus\{0\}}\left(f(x+z)-f(x)-\nabla f(x)\cdot z\mathbf{1}_{\{|z|\leqslant 1\}}\right)\frac{C_{d,\alpha}}{|z|^{d+\alpha}} \d z+\<Ax,\nabla f(x)\>,$$
and
$$\left\{ f \in C^2(\mathbb{R}^d): \int_{\{|z|>1\}}[f(x+z)-f(x)]\frac{1}{|z|^{d+\alpha}}\d z<\infty, \ \text{for} \ x\in \mathbb{R}^d \right\}\subset D_w(L).$$
By \cite[Theorem 3]{wj2012}, if  the real parts of all the eigenvalues of $A$ are negative, then the process is exponentially ergodic. We will prove the non-strong ergodicity by Theorem \ref{inequality-type2}.
\begin{theorem}\label{stable}
Ornstein-Uhlenbeck process $(X_t)_{t\geqslant 0}$ driven by symmetric $\alpha-$stable noise is not strongly ergodic.
\end{theorem}
\begin{proof}
Let $u(x)=\log (|x|+1)$ and $v(x)=|x|^{\theta}$, where $\theta\in (0,1\wedge \alpha)$. It's easy to check that  $u(x), \ v(x)\in D_w(L)$ and  $\lim\limits_{|x|\rightarrow\infty}{u(x)}/{v(x)}=0$. Moreover, there exists $r_0>0$ such that $Lv(x)\leqslant\mathbf{1}_{B(0,r_0)}$. See \cite{wj2013}.

To apply Theorem \ref{inequality-type2}, we need only prove that there exists $C(d,\alpha)$ only depend on $d,\alpha,$ such that
$$Lu(x)\geqslant -C(d,\alpha), \ \text{for}\ |x| \ \text{large enough}.$$

First, we estimate the drift coefficient
\begin{equation}\label{drift}
\<\nabla u(x),Ax\>=\frac{\<x,Ax\>}{|x|(|x|+1)}\geqslant -\left(\sum_{i,j=1}^{d}a_{ij}^2\right)^{\frac{1}{2}}.
\end{equation}

Next we turn to estimate the fractional Laplacian for $|x|$ large enough.
\[
\begin{split}
-(-\Delta)^{\alpha/2}u(x)=&\int_{\mathbb{R}^d\setminus\{0\}}\left(u(x+z)-u(x)-\nabla u(x)\cdot z\mathbf{1}_{\{|z|\leqslant 1\}}\right)\frac{C_{d,\alpha}}{|z|^{d+\alpha}} \d z\\
\geqslant&\int_{\{|z|\leqslant 1\}}\left(u(x+z)-u(x)-\nabla u(x)\cdot z\right)\frac{C_{d,\alpha}}{|z|^{d+\alpha}} \d z\\
&+\int_{\{|z|>1 \}}(u(x+z)-u(x))\frac{C_{d,\alpha}}{|z|^{d+\alpha}} \d z\\
=:&A(x)+B(x).
\end{split}
\]
For $|x|> 1,$
\[
\begin{split}
A(x)&=\frac{1}{2}\int_{\{|z|\leqslant 1\}}\left[z^{\mathrm{T}}\mathrm{Hess}(u(x+\theta z))z\right]\frac{C_{d,\alpha}}{|z|^{d+\alpha}} \d z\\
&\geqslant \frac{1}{2} C_{d,\alpha}\left(\frac{1}{(|x|+2)(|x|+1)}-\frac{2}{(|x|-1)^2}-\frac{1}{(|x|-1)^3}\right)\int_{\{|z|\leqslant 1\}}\frac{1}{|z|^{d+\alpha-2}}\d z\\
&\rightarrow 0 \ \text{as} \ |x|\rightarrow\infty.
\end{split}
\]

The calculation of $B(x)$ is complicated, and is divided in three cases: $\alpha\in (1,2)$, $\alpha=1$ and $\alpha\in (0,1)$.

\textbf{Case 1:} $\mathbf{\alpha\in (1,2)}.$
Using Taylor's formula to $u(x),$ we have
\[
\begin{split}
B(x)&=\int_{\{|z|\geqslant1 \}}(u(x+z)-u(x))\frac{C_{d,\alpha}}{|z|^{d+\alpha}} \d z\\
&=\int_{\{|z|\geqslant1 \}}\frac{\<x+\theta z,z\>}{(|x+\theta z|+1)(|x+\theta z|)}\frac{C_{d,\alpha}}{|z|^{d+\alpha}} \d z\\
&\geqslant -\int_{\{|z|\geqslant1 \}}\frac{1}{|x+\theta z|+1}\frac{C_{d,\alpha}}{|z|^{d+\alpha-1}} \d z\\
&\geqslant-C(d)\int_{1}^{\infty}\frac{1}{r^{\alpha}}\d r=-\frac{C(d)}{\alpha-1}.
\end{split}
\]
So there exists $R_1$ large enough such that $-(-\Delta)^{\alpha/2}u(x)\geqslant-\frac{2C(d)}{\alpha-1} \ \text{for}\ |x|>R_1.$

\textbf{Case 2:} $\mathbf{\alpha=1}.$ Using integration by parts formula,
\[
\begin{split}
B(x)&=\int_{\{1<|z|<|x| \}}(u(x+z)-u(x))\frac{C_{d,\alpha}}{|z|^{d+\alpha}} \d z+\int_{\{|z|\geqslant|x| \}}(u(x+z)-u(x))\frac{C_{d,\alpha}}{|z|^{d+\alpha}} \d z\\
&=C(d)\int_1^{|x|}\log\left(1-\frac{r}{|x|+1}\right)\frac{1}{r^2} \d r+C(d)\int_{|x|}^{\infty}\log\left(\frac{(r-|x|+1)}{|x|+1}\right)\frac{1}{r^2} \d r\\
&=C(d)\left[\frac{1}{|x|}\log|x|-\frac{2\log|x|}{|x|+1}-\frac{1}{|x|}\log|x|+\frac{\log|x|}{|x|-1}\right]\\
&\rightarrow 0 \ \text{as}\  |x|\rightarrow\infty.
\end{split}
\]
Thus there exists $R_2$ large enough such that $-(-\Delta)^{\alpha/2}u(x)\geqslant-1$, $\text{for}\ |x|>R_2.$

\textbf{Case 3:} $\mathbf{\alpha\in (0,1)}.$ First we divide $B(x)$ in three parts:
\[
\begin{split}
B(x)=&\int_{\{1<|z|<|x|-1 \}}(u(x+z)-u(x))\frac{C_{d,\alpha}}{|z|^{d+\alpha}} \d z\\
&+\int_{\{|x|-1\leqslant|z|\leqslant|x|+1 \}}(u(x+z)-u(x))\frac{C_{d,\alpha}}{|z|^{d+\alpha}} \d z\\
&+\int_{\{|z|>|x|+1 \}}(u(x+z)-u(x))\frac{C_{d,\alpha}}{|z|^{d+\alpha}} \d z\\
:=&I_1(x)+I_2(x)+I_3(x).
\end{split}
\]
We estimate $I_1, I_2$ and $I_3$ one by one. First, we have
\[
\begin{split}
I_1&=\int_{\{1<|z|\leqslant |x|-1\}}\big(\log(|x+z|+1)-\log(|x|+1)\big)\frac{C_{d,\alpha}}{|z|^{d+\alpha}} \d z\\
&\geqslant \int_{\{1<|z|\leqslant |x|-1\}}\log\left(\frac{(|x|-|z|+1)}{|x|+1}\right)\frac{C_{d,\alpha}}{|z|^{d+\alpha}} \d z\\
&= C(d)\int_1^{|x|-1}\log\left(1-\frac{r}{|x|+1}\right)\frac{1}{r^{1+\alpha}} \d r\\
&\geqslant\frac{C(d)}{\alpha}\left[\log\left(\frac{|x|}{|x|+1}\right)-\frac{1}{(|x|-1)^{\alpha}}\log\left(\frac{2}{|x|+1}\right)-\int_0^{|x|-1}\frac{1}{r^{\alpha}}\frac{1}{|x|+1-r}\d r\right].\\
\end{split}
\]
To estimate the last integral above, we use the integral representation of Gauss hypergeometric function $F(a,b,c,z)$ (see \cite[15.3.1]{Abramowitz84}):
\begin{equation}\label{integral hypergeometric function}
F(a,b,c,z)=\frac{\Gamma(c)}{\Gamma(b)\Gamma(c-b)}\int_{0}^{1}t^{b-1}(1-t)^{c-b-1}(1-tz)^{-a}\d t, \ \mathrm{Re}(c)>\mathrm{Re}(b)>0.
\end{equation}
By using (\ref{integral hypergeometric function}), we have
\[
\begin{split}
-\int_0^{|x|-1}\frac{1}{r^{\alpha}}\frac{1}{|x|+1-r}\d r&=-\frac{1}{(1-\alpha)(|x|-1)^{\alpha-1}(|x|+1)}F\left(1,1-\alpha,2-\alpha,\frac{|x|-1}{1+|x|}\right)\\
&=:J_1(x).
\end{split}
\]
Similarly,
\[
\begin{split}
I_3(x)&=\int_{\{|z|\geqslant |x|+1\}}\big(\log(|x+z|+1)-\log(|x|+1)\big)\frac{C_{d,\alpha}}{|z|^{d+\alpha}} \d z\\
&\geqslant \int_{\{|z|\geqslant |x|+1\}}\log\left(\frac{(|z|-|x|+1)}{|x|+1}\right)\frac{C_{d,\alpha}}{|z|^{d+\alpha}} \d z\\
&= C(d)\int_{|x|+1}^{\infty}\log\left(\frac{(r-|x|+1)}{|x|+1}\right)\frac{1}{r^{1+\alpha}} \d r\\
&=\frac{C(d)}{\alpha(|x|+1)^{\alpha}}\log\left(\frac{2}{|x|+1}\right)+\frac{C(d)}{\alpha}\int_{|x|+1}^{\infty}\frac{1}{r^{\alpha}}\frac{1}{1-|x|+r}\d r\\
\end{split}
\]
and
\[
\begin{split}
\int_{|x|+1}^{\infty}\frac{1}{r^{\alpha}}\frac{1}{1-|x|+r}\d r&=\frac{1}{\alpha(|x|+1)^{\alpha}}F\left(1,\alpha,1+\alpha,\frac{|x|-1}{|x|+1}\right)=:J_2(x).\\
\end{split}
\]
Next, we calculate
\[
\begin{split}
J_1(x)+J_2(x)=&\frac{F\left(1,1-\alpha,2-\alpha,\frac{|x|-1}{1+|x|}\right)}{(\alpha-1)(|x|-1)^{\alpha-1}(|x|+1)}+\frac{F\left(1,\alpha,1+\alpha,\frac{|x|-1}{|x|+1}\right)}{\alpha(|x|+1)^{\alpha}}\\
=&\left[\frac{F\left(1,1-\alpha,2-\alpha,\frac{|x|-1}{1+|x|}\right)}{(\alpha-1)(|x|-1)^{\alpha-1}(|x|+1)}-\frac{F\left(1,1-\alpha,2-\alpha,\frac{|x|-1}{1+|x|}\right)}{(\alpha-1)(|x|+1)^{\alpha}}\right]\\
&+\left[\frac{F\left(1,1-\alpha,2-\alpha,\frac{|x|-1}{1+|x|}\right)}{(\alpha-1)(|x|+1)^{\alpha}}+\frac{F\left(1,\alpha,1+\alpha,\frac{|x|-1}{|x|+1}\right)}{\alpha(|x|+1)^{\alpha}}\right]\\
=:&K_1(x)+K_2(x).
\end{split}
\]
According to \cite[(3.18)--(3.21)]{sandric2013},
$$\lim\limits_{|x|\rightarrow\infty}(1+|x|)^{\alpha}(K_1(x)+K_2(x))=\pi \cot \left(\pi\alpha/2\right).$$
Hence
\[
\begin{split}
I_1(x)+I_3(x)\geqslant&\frac{C(d)}{\alpha}\left[\log\left(\frac{|x|}{|x|+1}\right)+\frac{1}{(|x|-1)^{\alpha}}\log\left(\frac{2}{|x|+1}\right)\right]\\
&+\frac{C(d)}{\alpha(|x|+1)^{\alpha}}\log\left(\frac{2}{|x|+1}\right)+\frac{C(d)}{\alpha}(K_1(x)+K_2(x))\\
\rightarrow &0 \ \text{as}\  |x|\rightarrow\infty.
\end{split}
\]
Next, we consider $I_2(x). $ Since
\[
\begin{split}
\int_{|x|-1}^{|x|+1}\log\left(\frac{|r-|x||+1}{|x|+1}\right)\frac{1}{r^{1+\alpha}} \d r&\geqslant \frac{\log(|x|+1)}{\alpha}\left[\frac{1}{(|x|+1)^{\alpha}}-\frac{1}{(|x|-1)^{\alpha}}\right]  \\
  &\rightarrow 0 \ \text{as}\  |x|\rightarrow\infty,
\end{split}
\]
we have $\lim\limits_{|x|\rightarrow\infty}I_2(x)=0.$
Thus for $\alpha\in (0,1)$,
$$\lim\limits_{|x|\rightarrow\infty}-(-\Delta)^{\alpha/2}u(x)\geqslant \lim\limits_{|x|\rightarrow\infty}(I_1(x)+I_2(x)+I_3(x))=0.$$
Therefore, we choose $R_3>1$ large enough such that $-(-\Delta)^{\alpha/2}u(x)\geqslant-1$ for any $|x|>R_3.$

Finally, we obtain that for each $\alpha\in (0,2)$, there exists a positive number $R_3$ large enough, such that $-(-\Delta)^{\alpha/2}u(x)\geqslant-c, $ where $c$ is a positive number.

Therefore, by combining (\ref{drift}) and Case 1--3, we prove that there exists $C(d,\alpha)$ only depending on $d,\alpha,$ such that $Lu(x)\geqslant -C(d,\alpha), \ \text{for}\ |x| \ \text{large enough}.$
\end{proof}

\begin{corollary}\label{nd simple atable example}
Consider the following stochastic differential equation driven by $\alpha$-stable noise on $\mathbb{R}^d$:
$$\d X_t = b(X_t) \d t + \d Z_t, \  X_0 = x_0,$$
where $b(x)=-x|x|^{\delta}, \ \delta\geqslant 0$. If $\delta>0$, \cite[Example 1.2]{wj2013} has proved the process is strongly ergodic. If $\delta=0,$ then by Theorem \ref{stable}, the process is non-strongly ergodic. Thus the process is strongly ergodic if and only if $\delta>0$.
\end{corollary}

\begin{remark}
According to Corollary \ref{nd simple atable example} and Example \ref{nd simple example}, we know that for SDE driven by symmetric $\alpha$-stable process ($\alpha\in (0,2]$) with polynomial drift $b(x)=-x|x|^{\delta}$, the strong ergodicity is independent of $\alpha$.
\end{remark}





\acks
The authors  thank reviewers for their valuable comments and helpful references on the first version of the paper.
The authors also would like to thank Professor Mu-Fa Chen for his valuable suggestions. This work is supported in part by the National Natural Science Foundation of China (Grant No.11771046 and No.11771047), the project from the Ministry of Education in China.


%
%
%
%

\end{document}